\documentclass[11pt]{article}

\newcommand{\reals}{\mbox{$\mathbb R$}}

\newcommand{\nats}{\mbox{$\mathbb N$}}
\newcommand{\ints}{\mbox{$\mathbb Z$}}

\newcommand{\comment}[1]{}

\usepackage{amsmath}  
\usepackage{amssymb}

\def\squarebox#1{\hbox to #1{\hfill\vbox to #1{\vfill}}}
\def\qed{\hspace*{\fill}
        \vbox{\hrule\hbox{\vrule\squarebox{.667em}\vrule}\hrule}\smallskip}
\newenvironment{proof}{\begin{trivlist}
  \item[\hspace{\labelsep}{\em\noindent Proof.~}]
  }{\qed\end{trivlist}}

\newtheorem{lemma}{Lemma}[section]
\newtheorem{theorem}[lemma]{Theorem}

\newtheorem{proposition}[lemma]{Proposition}
\newtheorem{claim}[lemma]{Claim}
\newtheorem{observation}[lemma]{Observation}

\def\squareforqed{\hbox{\rlap{$\sqcap$}$\sqcup$}}
\def\qed{\ifmmode\squareforqed\else{\unskip\nobreak\hfil
\penalty50\hskip1em\null\nobreak\hfil\squareforqed
\parfillskip=0pt\finalhyphendemerits=0\endgraf}\fi}

\newlength{\tablength}
\newlength{\spacelength}
\newcommand{\tabstar}{\hspace*{\tablength}}
\newcommand{\spacestar}{\hspace*{\spacelength}}
\def\obeytabs{\catcode`\^^I=\active}
{\obeytabs\global\let^^I=\tabstar}
{\obeyspaces\global\let =\spacestar}
\newenvironment{display}{\begingroup\obeylines\obeyspaces\obeytabs}{\endgroup}
\newenvironment{prog}{\begin{display}\parskip0pt\sf}{\end{display}}

\title{Induced subgraphs of hypercubes}

\author{
{\sl Geir Agnarsson}
\thanks{Department of Mathematical Sciences,
George Mason University, 
MS 3F2, 
4400 University Drive, 
Fairfax, VA 22030, 
{\tt geir@math.gmu.edu}}} 

\date{}

\begin{document}

\maketitle

\begin{abstract}
Let $Q_k$ denote the $k$-dimensional hypercube 
on $2^k$ vertices. A vertex in a subgraph of $Q_k$ 
is {\em full} if its degree is $k$. We apply the Kruskal-Katona
Theorem to compute the maximum number of full vertices 
an induced subgraph on $n\leq 2^k$ vertices of $Q_k$ 
can have, as a function of $k$ and $n$. This is then used
to determine $\min(\max( |V(H_1)|, |V(H_2)|))$
where (i) $H_1$ and $H_2$ are induced subgraphs of $Q_k$, and
(ii) together they cover all the edges of $Q_k$, that is
$E(H_1)\cup E(H_2) = E(Q_k)$.

\vspace{3 mm}


\vspace{2 mm}

\end{abstract}

\section{Introduction}
\label{sec:intro}

The maximum number $f(n)$ of edges of an induced subgraph on $n$ vertices
of the hypercube $Q_k$, where $k \geq \lceil \lg n \rceil$, 
has been studied extensively in~\cite{S.Hart-DM-note}, 
\cite{Mcllroy-SIAM-binary}, ~\cite{Greene-Knuth}, ~\cite{Delange-sur-la},
and~\cite{HyperCube} to name a few articles. The function $f(n)$
satisfies, and is determined by, the well-known divide-and-conquer 
maximin recurrence
\begin{equation}
\label{eqn:max-f}
f(n) = \max_{\substack{n_1 + n_2 = n \\ 
                          n_1, n_2 \geq 1}} 
               \left( \min(n_1,n_2) + f(n_1) + f(n_2)\right),
\end{equation}
and can be expressed compactly by the formula $f(n) = \sum_{i=0}^{n-1}s(i)$,
where $s(i)$ is the sum of the digits of $i$ when expressed
as a binary number. The function $f$ and its number sequence 
$(f(n))_{n=0}^{\infty} = (0,1,2,4,5,7,9,12,13,15,17,20,\ldots)$
is given in~\cite[A000788]{sum-of-1s}, where it is presented by a 
different recursion. The divide-and-conquer maximin recurrence 
(\ref{eqn:max-f}) 
is one of the most studied recurrences, especially since it
occurs naturally when analysing wort-case scenarios in sorting 
algorithms~\cite{Li-Reingold-maximin}. The maximin recurrence 
(\ref{eqn:max-f}) is also one of the few such maximin recurrences
that have a solution $f(n)$ that can be expressed explicitly by a
formula.

Clearly the hypercube $Q_k$ is a subgraph of the $k$-dimensional
rectangular grid graph ${\ints}^k$. It is interesting to note that 
for $k \geq \lceil \lg n \rceil$ the maximum number of edges of an 
induced subgraph on $n$ vertices of ${\ints}^k$ is the same
if we restrict to $Q_k$, namely $f(n)$. However, if we consider
$k$ fixed and consider the maximum number $g_k(n)$ of edges of an induced
subgraph on $n$ vertices of the grid graph ${\ints}^k$, then
the only cases where a formula for $g_k(n)$ is known is for $k\in\{1,2\}$:
trivially $g_1(n) = n-1$, and $g_2(n) = \lceil 2n - 2\sqrt{n}\rceil$
as proved in~\cite{EX-animals}. For $k\geq 3$ no formula for $g_k(n)$ is
known, but the first few terms of 
$(g_3(n))_{n=1}^{\infty} = (0,1,2,4,5,7,9,12,13,15,17,20,\ldots)$
is given heuristically in~\cite[A007818]{wrong}. -- In short, 
considering $k$ fixed (and hence not allowing conveniently 
large dimensions) makes it harder to solve such maximin problems.

The purpose of this article is to consider a related problem of induced
subgraphs on $n$ vertices of the hypercube $Q_k$ where we
consider $k$ fixed. A vertex of a subgraph of $Q_k$ is called {\em full}
in the subgraph if its degree is $k$. If we let $\phi_k(n)$ be the maximum
number of full vertices an induced subgraph on $n$ vertices
of $Q_k$ can have, then (i) we show that $\phi_k(n)$ satisfies 
a divide-and-conquer maximin recurrence (\ref{eqn:feqn}), and (ii) 
we derive its solution, namely the formula for $\phi_k(n)$ given 
in Theorem~\ref{thm:main-full}.
We then apply the formula for $\phi_k(n)$ to (iii) determine 
the min-max function $\min(\max( |V(H_1)|, |V(H_2)|))$
where both $H_1$ and $H_2$ are induced subgraphs of $Q_k$, and
together they cover all the edges of $Q_k$. We show that 
this min-max function is given by the formula in 
Theorem~\ref{thm:min-max-cover}.

The remainder of the paper is organized as follows. 

In Section~\ref{sec:UBF} we recall the celebrated Katona-Kruskal
Theorem that describes when exactly an integral vector of ${\ints}^{d+1}$
is an $f$-vector of a $(d-1)$-dimensional simplicial complex.
We then derive some helpful tools: Claim~\ref{clm:sum-less} and
Lemmas~\ref{lmm:1},~\ref{lmm:2}, and~\ref{lmm:3}, that we will use
in the following section.

In Section~\ref{sec:main} we use what we have derived in Section~\ref{sec:UBF}
to derive our main Theorem~\ref{thm:main-full} that determines the
exact maximum number of full vertices an induced subgraph on $n$ vertices
of $Q_k$ can have.

In the final Section~\ref{sec:appl} we apply Theorem~\ref{thm:main-full}
from the previous section prove Theorem~\ref{thm:min-max-cover}, that
determines $\min(\max( |V(H_1)|, |V(H_2)|))$, the function of $k\in\nats$  
where $H_1$ and $H_2$ are induced subgraphs of $Q_k$, and
together $H_1$ and $H_2$ cover all the edges of $Q_k$. 

\paragraph{Notation and terminology}
The set of integers will be denoted by $\ints$ and
the set of natural numbers $\{1,2,3,\ldots\}$ by $\nats$.
For $n\in\nats$ let $[n] = \{1,\ldots,n\}$.
For a set $X$ denote the set of all subsets of $X$ by
$2^X$. Denote the subsets of $X$ of 
cardinality $i$ by $\binom{X}{i}$, so for $X$ finite
we have $\left|\binom{X}{i}\right| = \binom{|X|}{i}$.
For ${\cal{S}}\subseteq 2^X$ and $y\not\in X$,
let ${\cal{S}}\uplus \{y\} = \{S \cup \{y\} : S\in {\cal{S}}\}$.

Unless otherwise stated, all graphs in this article
will be finite, simple and undirected. For a graph $G$, its
set of vertices will be denoted by $V(G)$ and its set
of edges by $E(G)$. Clearly $E(G)\subseteq \binom{V(G)}{2}$
the set of all 2-element subsets of $V(G)$. We will denote
an edge with endvertices $u$ and $v$ by $uv$ instead of the
actual 2-set $\{u,v\}$. 
By an {\em induced subgraph} $H$ of $G$ we mean a subgraph
$H$ such that $V(H)\subseteq V(G)$ in the usual set theoretic
sense, and such that if $u,v\in V(H)$ and $uv\in E(G)$, then
$uv\in E(H)$. If $U\subseteq V(G)$ then the subgraph of $G$ 
induced by $V$ will be denoted by $G[U]$. 

\comment{ 
For $k\in\nats$ a {\em rectangular grid}
${\ints}^k$ in our context is a infinite graph with the point set
${\ints}^k$ as its vertices and where two points
$\tilde{x} = (x_1,\ldots,x_k)$ and $\tilde{y} = (y_1,\ldots,y_k)$
are connected by an edge iff the {\em Manhattan distance}
$d(\tilde{x},\tilde{y}) = \sum_{i=1}^k |x_i - y_i| = 1$.
So, two points are connected iff they only differ in one
coordinate, in which they differ by $\pm 1$. This Manhattan
distance measure is the metric corresponding to the
{\em 1-norm}  $\|\tilde{x}\|_1 = \sum_{i=1}^k|x_i|$ in
the $k$-dimensional Euclidean space ${\reals}^k$. The {\em hypercube} $Q_k$
is then the subgraph of the grid ${\ints}^k$ induced by
the $2^k$ points $\{0,1\}^k$, so in graph theoretic terms here above
we have $Q_k = {\ints}^k[\{0,1\}^k]$. 
}
 
For $k\in\nats$ the {\em hypercube}
$Q_k$ in our context is a simple graph with the $2^k$
vertices $\{0,1\}^k$, and where two vertices 
$\tilde{x},\tilde{y}\in \{0,1\}^k$ are adjacent
iff the {\em Manhattan distance}
$d(\tilde{x},\tilde{y}) = \sum_{i=1}^k |x_i - y_i| = 1$.
So, two vertices are connected iff they only differ in one
coordinate, in which they differ by $\pm 1$. The vertices of the 
hypercube $Q_k$ are more commonly viewed as binary strings of length $k$ 
instead of actually points in the $k$-dimensional Euclidean space. 
In that case the Manhattan distance is called the called the 
{\em Hamming distance}. 
We will not make a specific distinction between these two slightly 
different presentations  of the hypercube $Q_k$. In many situations
it will be convenient to partition the hypercube $Q_k$ into two
copies of $Q_{k-1}$ where corresponding vertices in each copy
are connected by and edge. If $b\in\{0,1\}$ and 
$B_b= \{\tilde{x}\in \{0,1\}^k : x_k = b\}$
is the set of binary strings of length $k$ with $k$-th bit equal to $b$,
then clearly each of $Q_{k-1}^0 := Q_k[B_0]$ and $Q_{k-1}^1 := Q_k[B_1]$ 
are induced subgraphs isomorphic to $Q_{k-1}$, and
(i) $V(Q_k) = V(Q_{k-1}^0) \cup V(Q_{k-1}^1) = B_0\cup B_1$ is a partition
and (ii) $E(Q_k) = E(Q_{k-1}^0) \cup E(Q_{k-1}^1) \cup C_{k-1}$
is also a partition of the edges where
\[
C_{k-1} = \{ \{(\tilde{x},0),(\tilde{x},1)\} : \tilde{x} \in V(Q_{k-1})\}.
\]
For $b\in\{0,1\}$ and $\tilde{x}\in V(Q_{k-1})$,
the {\em copy} of $(\tilde{x},b)\in V(Q_{k-1}^b)$ is 
the vertex $(\tilde{x},1-b)\in V(Q_{k-1}^{1-b})$, and these
well be referred as {\em copies}. 
This decomposition of $Q_k$ will be denoted by 
$Q_k = Q_{k-1}^0 \boxplus Q_{k-1}^1$.

\section{Some properties of the upper boundary function}
\label{sec:UBF}

The following proposition on the {\em binomial representation} of an 
integer is stated in~\cite{Stanley} and in~\cite{Hibi},
and a simple proof by greedy algorithm can be found in the latter citation.
\begin{proposition}
\label{prp:binom-rep}
For $m, i\in \nats$ there is a {\em unique binomial representation (UBR)} 
of $m$ as 
\begin{equation}
\label{eqn:binom-rep}
m = \binom{n_i}{i} + \binom{n_{i-1}}{i-1} + \cdots + \binom{n_j}{j}
\end{equation}
where $n_i > n_{i-1} > \cdots > n_j \geq j\geq 1$. 
\end{proposition}
For $m, i\in \nats$ one can use the UBR 
to define the {\em upper $i$-boundary of $m$} 
\[
m^{(i)} = \binom{n_i}{i+1} + \binom{n_{i-1}}{i} + \cdots + \binom{n_j}{j+1}.
\]
\begin{proposition}
\label{prp:increasing}
For a fixed $i\in\nats$ the function $m\mapsto m^{(i)}$ is increasing.
\end{proposition}
\begin{proof}
For $m,i\in\nats$ consider the UBR of $m$ as in (\ref{eqn:binom-rep}).

If $j\geq 2$, then 
\[
m+1 = \binom{n_i}{i} + \binom{n_{i-1}}{i-1} + \cdots + \binom{n_j}{j}
+\binom{j-1}{j-1}
\]
is the UBR of $m+1$ and so
$(m+1)^{(i)} = m^{(i)} + \binom{j-1}{j} = m^{(i)}$.

Otherwise $j=1$, and hence there is a largest index $\ell\in[i]$
such that $n_h = n_1+h-1$ for all $h\in\{1,\ldots,\ell\}$. In this
case we have $n_{\ell + 1} > n_{\ell}+1 = n_1 + \ell$ and 
\begin{eqnarray*}
m+1 & = & 
\binom{n_i}{i} + \binom{n_{i-1}}{i-1} + \cdots + \binom{n_{\ell+1}}{\ell+1} 
+ \binom{n_1 +\ell -1}{\ell} + \ldots + \binom{n_1}{1} + 1 \\
   & = & 
\binom{n_i}{i} + \binom{n_{i-1}}{i-1} + \cdots + \binom{n_{\ell+1}}{\ell+1}
+ \binom{n_1+\ell}{\ell}
\end{eqnarray*}
and hence 
\[
(m+1)^{(i)} - m^{(i)} = \binom{n_1+\ell}{\ell+1} 
- \left[\binom{n_1+\ell-1}{\ell+1} + \cdots + \binom{n_1}{2}\right] = 
\binom{n_1}{1} = n_1.
\]
\end{proof}
We see from the above proof when exactly the function $m\mapsto m^{(i)}$ 
is strictly increasing; namely, whenever the last binomial coefficient
in the UBR of $m$ has the form $\binom{n_1}{1}$, then 
$(m+1)^{(i)} = m^{(i)} + n_1 > m^{(i)}$. In particular, for $i < n$ we
have 
\[
\binom{n}{i} - 1 = \binom{n-1}{i} + \binom{n-2}{i-1} + \cdots + \binom{n-i}{1}
\]
and hence the following observation.
\begin{observation}
\label{obs:strict}
For $i,n\in\nats$ with $i<n$ then
\[
\left(\binom{n}{i} - 1\right)^{(i)} = \binom{n}{i}^{(i)} - (n-i) 
< \binom{n}{i}^{(i)}.
\]
\end{observation}

In this article the {\em $f$-vector} of a $(d-1)$-dimensional 
simplicial complex $\Delta$ will be given by 
$\tilde{f}(\Delta) = \tilde{f} = (f_{-1},f_0,f_1,\ldots, f_{d-1}) 
\in {\ints}^{d+1}$
where $f_i = f_i(\Delta)$ denotes the number of $i$-dimensional faces of 
$\Delta$. For convenience we include the empty face $\emptyset$ in $\Delta$.
Since by convention $\dim(\emptyset) = -1$ then
we always have $f_{-1} = 1$. 
The following celebrated result proved independently by 
Kruskal~\cite{Kruskal}, Katona~\cite{Katona} and 
Sch\"{u}tzenberger~\cite{Shuetz}, is usually called
the {\em Kruskal-Katona Theorem}, since it was not realized
at first that Sch\"{u}tzenberger had the first proof.
It is sometimes called the {\em KKS Theorem} for short.
\begin{theorem} 
\label{thm:KKS}
An integral vector 
$\tilde{f} = (f_{-1},f_0,f_1,\ldots, f_{d-1}) \in {\ints}^{d+1}$
is an $f$-vector of a $(d-1)$-dimensional simplicial complex $\Delta$ 
if and only if $0 < f_i \leq f_{i-1}^{(i)}$ for each $i\in \{1,\ldots,d-1\}$.
\end{theorem}
Although we will not regurgitate the proof of Theorem~\ref{thm:KKS}
here, a few comments about it will be useful for us here
in this section. -- Note that a simplicial complex $\Delta$
on vertices $V = \{v_1,\ldots,v_n\}$ can be viewed 
as an abstract simplicial complex; a collection of subsets
of $[n]$ satisfying 
(1) $\{i\} \in \Delta$ for each $i\in [n]$, and 
(2) $F\subseteq G \in \Delta \Rightarrow F\in \Delta$. 
For each $i$ we can linearly order the $i$-element subsets of $\nats$ 
in the {\em reverse lexicographical order}. So for $i=3$ the order
would start as follows:
\[
\begin{array}{l}
\{1,2,3\}, \{1,2,4\}, \{1,3,4\}, \{2,3,4\}, \{1,2,5\}, \{1,3,5\}, \{2,3,5\}, \\
\{1,4,5\}, \{2,4,5\}, \{3,4,5\}, \{1,2,6\}, \{1,3,6\}, \{2,3,6\}, \ldots
\end{array}
\]
For an integral vector 
$\tilde{f} = (f_{-1},f_0,f_1,\ldots, f_{d-1}) \in {\ints}^{d+1}$
let $\Delta_{\tilde{f}}\subseteq 2^{\nats}$ consist of the first $f_{i-1}$ 
$i$-element subsets of $\nats$ in the
reverse lexicographical ordering for each $i\in\{0,1,\ldots,d\}$. The proof
of Theorem~\ref{thm:KKS} is based on proving the equivalence of
the following three statements~\cite{Stanley}.
\begin{enumerate}
  \item The integral vector $\tilde{f}$ is an $f$-vector of a simplicial 
complex $\Delta$.
  \item $\Delta_{\tilde{f}}$ is a simplicial complex.
  \item $f_i \leq f_{i-1}^{(i)}$ for each $i\in\{1,\ldots,d-1\}$.
\end{enumerate}
The hard part of the proof is the implication $1\Rightarrow 2$.

For a fixed $i$ we have a well-defined function $m\mapsto m^{(i)}$
which we will refer to as 
the {\em upper boundary function}\footnote{as a function this
has been called the ``pseudo power function''~\cite{jim}. It is similar
to the ``upper boundary operator''~\cite{Grunbaum}.} or the 
{\em UB} function for short.
The remainder of this section will be devoted to the derivation
of some properties of the UB function.
We will, in part, use the above equivalence to prove these properties.

For one such property of the UB function, let $i,m_1,m_2,N\in\nats$
be such that $m_1,m_2\leq\binom{N}{i}$ and consider
two integral vectors in ${\ints}^{i+2}$
\begin{eqnarray*}
\tilde{f}_1 & = & 
\left(\binom{N}{0},\ldots,\binom{N}{i-1},m_1,m_1^{(i)}\right), \\
\tilde{f}_2 & = & 
\left(\binom{N}{0},\ldots,\binom{N}{i-1},m_2,m_2^{(i)}\right).
\end{eqnarray*}
By Theorem~\ref{thm:KKS} both $\Delta_{\tilde{f}_1}$ and $\Delta_{\tilde{f}_1}$
are simplicial complexes. Assume we have disjoint representations
$\Delta_{\tilde{f}_1}\subseteq 2^{[N]}$ and 
$\Delta_{\tilde{f}_2}\subseteq 2^{[2N]\setminus [N]}$ (where
in the latter representation $N$ has been added to each element
of each set in $\Delta_{\tilde{f}_2}$) and let 
$\Delta := \Delta_{\tilde{f}_1}\cup \Delta_{\tilde{f}_2}$ be their 
union. By definition $\Delta$ is clearly a simplicial complex with
$f_{i-1}(\Delta) = m_1+m_2$ and $f_i(\Delta) = m_1^{(i)}+m_2^{(i)}$. 
By Theorem~\ref{thm:KKS} we have the following.
\begin{claim}
\label{clm:sum-less}
$(m_1+m_2)^{(i)} \geq m_1^{(i)}+m_2^{(i)}$.
\end{claim}

Let ${\cal{F}}_i(N)$ denote the first $N$ elements of of $\binom{\nats}{i}$.
We clearly have then (i) $|{\cal{F}}_i(N)| = N$ for $i\geq 1$,
(ii) ${\cal{F}}_i(N_1) \subseteq {\cal{F}}_i(N_2)$ iff $N_1\leq N_2$,
(iii) ${\cal{F}}_i\left(\binom{k}{i}\right) = \binom{[k]}{i}$, and
by definition of $\Delta_{\tilde{f}}$ here above we have
for an $f$-vector $\tilde{f}$ of a simplicial complex that (iv) 
\[
\Delta_{\tilde{f}} = 
{\cal{F}}_0(f_{-1})\cup{\cal{F}}_1(f_{0})\cup\cdots\cup{\cal{F}}_d(f_{d-1}).
\]
(v) Finally note that if $|X| = i$ and $|Y| = i+1$, then 
by Theorem~\ref{thm:KKS} we have that
$X\subseteq Y\in {\cal{F}}_{i+1}(N^{(i)})$ implies that $X\in {\cal{F}}_i(N)$.

For $N\in\nats$, $m_1\leq\binom{N}{i}$ and $m_2\leq \binom{N}{i-1}$
let $\mu = \min(m_2^{(i-1)},m_1)$ and consider two integral vectors
\begin{eqnarray*}
\tilde{f}_1 & = & 
\left(\binom{N}{0},\binom{N}{1}\ldots,\binom{N}{i-1},m_1,m_1^{(i)}\right) 
\in{\ints}^{i+2}, \\
\tilde{f}_2 & = & 
\left(\binom{N}{0},\ldots,\binom{N}{i-2},m_2,\mu\right)
\in{\ints}^{i+1}.
\end{eqnarray*}
By the Theorem~\ref{thm:KKS} both $\Delta_{\tilde{f}_1}$ and $\Delta_{\tilde{f}_2}$
are simplicial complexes of dimensions $i$ and $i-1$ respectively.
Assume we have abstract representations 
$\Delta_{\tilde{f}_1} \subseteq 2^{[N]}$ and 
$\Delta_{\tilde{f}_2} \subseteq 2^{[N]}$ and
let 
\begin{equation}
\label{eqn:D}
\Delta = \Delta_{\tilde{f}_1} \cup (\Delta_{\tilde{f}_2}\uplus \{N+1\})
\subseteq 2^{[N+1]}.
\end{equation}
\begin{claim}
\label{clm:D=sc}
$\Delta$ from (\ref{eqn:D}) is a simplicial complex of dimension $i$.
\end{claim}
\begin{proof}
For $\ell \in \{0,1,\ldots,i-1\}$ we have $\binom{[N+1]}{\ell}\subseteq \Delta$
and
\begin{eqnarray*}
\Delta\cap\binom{[N+1]}{i} & = & 
  {\cal{F}}_i(m_1)\cup ({\cal{F}}_{i-1}(m_2)\uplus \{N+1\}) \\
\Delta\cap\binom{[N+1]}{i+1} & = & 
  {\cal{F}}_{i+1}(m_1^{(i)})\cup ({\cal{F}}_i(\mu)\uplus \{N+1\}) 
\end{eqnarray*}
We only need to check 
$F\subseteq G\in \Delta\Rightarrow F\in\Delta$ for 
$F\in \Delta\cap\binom{[N+1]}{i}$ and $G\in \Delta\cap\binom{[N+1]}{i+1}$.
Here there are three cases to consider.

(a) $N+1\not\in F,G$:
Here we have $F\subseteq G\in {\cal{F}}_{i+1}(m_1^{(i)})$ and hence
$F\in {\cal{F}}_i(m_1)\subseteq \Delta$.

(b) $N+1\in F,G$:
Here $G\in {\cal{F}}_i(\mu)\uplus \{N+1\}$ and hence
$G\setminus\{N+1\} \in {\cal{F}}_i(\mu)\subseteq {\cal{F}}_i(m_2^{(i-1)})$
Since 
$F\setminus\{N+1\}\subseteq G\setminus\{N+1\}\in {\cal{F}}_i(m_2^{(i-1)})$,
we have $F\setminus\{N+1\}\in {\cal{F}}_{i-1}(m_2)$ and hence
$F\in {\cal{F}}_{i-1}(m_2)\uplus \{N+1\}\subseteq \Delta\cap\binom{[N+1]}{i}$ 
and so $F\in\Delta$.
 
(c) $N+1\not\in F$ and $N+1\in G$:
As in (b) we have $G\in {\cal{F}}_i(\mu)\uplus \{N+1\}$ and hence
$F = G\setminus\{N+1\} \in {\cal{F}}_i(\mu)\subseteq {\cal{F}}_i(m_1)
\subseteq \Delta\cap\binom{[N+1]}{i}$ and hence $F\in\Delta$.

Therefore we have $F\subseteq G\in \Delta\Rightarrow F\in\Delta$ for
all $F$ and $G$, and this completes the proof of the claim.
\end{proof}
For the $f$-vector of $\Delta$ in Claim~\ref{clm:D=sc} we have 
\[
f_{i-1}(\Delta) = 
|{\cal{F}}_i(m_1)| + |({\cal{F}}_{i-1}(m_2)\uplus \{N+1\})| = m_1 + m_2,
\]
and 
\[
f_i(\Delta) = 
|{\cal{F}}_i(m_1^{(i)})| + |({\cal{F}}_{i-1}(\mu)\uplus \{N+1\})| = 
m_1^{(i)} + \min(m_2^{(i-1)},m_1).
\]
By Theorem~\ref{thm:KKS} we obtain the following lemma
as a corollary.
\begin{lemma}
\label{lmm:1}
For $m_1,m_2,i\in\nats$ we have
\[
(m_1+m_2)^{(i)} \geq m_1^{(i)} + \min(m_2^{(i-1)},m_1).
\]
\end{lemma}

Let $m_1,m_2,i,N\in\nats$ be such that $m_1+m_2\leq\binom{N}{i-1}$.
By Lemma~\ref{lmm:1} we get
\[
\left(m_1+m_2+\binom{N}{i}\right)^{(i)}\geq
\binom{N}{i+1} + \min\left((m_1+m_2)^{(i-1)},\binom{N}{i}\right).
\]
By assumption and Proposition~\ref{prp:increasing}
we have $(m_1+m_2)^{(i-1)} \leq \binom{N}{i}$ and hence
\[
\min\left((m_1+m_2)^{(i-1)},\binom{N}{i}\right) = (m_1+m_2)^{(i-1)}.
\]
By Claim~\ref{clm:sum-less} we therefore have the following.
\begin{lemma}
\label{lmm:2}
For $m_1,m_2,i,N\in\nats$ with $m_1+m_2\leq\binom{N}{i-1}$
we have 
\[
\left(m_1+m_2+\binom{N}{i}\right)^{(i)}\geq 
m_1^{(i-1)} + m_2^{(i-1)} + \binom{N}{i+1}.
\]
\end{lemma}

Our final objective in this section is to prove
the following 
\begin{lemma}
\label{lmm:3}
If $0\leq m,m_1,m_2 \leq \binom{N}{i}$ and $m_1+m_2 = m +\binom{N}{i}$
then
\[
m_1^{(i)} + m_2^{(i)} \leq m^{(i)} + \binom{N}{i+1}.
\]
\end{lemma}
To prove Lemma~\ref{lmm:3} we let 
$\mathbf{P}(i)$ be the statement of Lemma~\ref{lmm:3} 
for a fixed $i\in\nats$.
\begin{quote}
$\mathbf{P}(i)$ : For all nonnegative integers $m,m_1,m_2,N$
that satisfy $0\leq m,m_1,m_2 \leq \binom{N}{i}$ 
and $m_1+m_2 = m +\binom{N}{i}$ we have
$m_1^{(i)} + m_2^{(i)} \leq m^{(i)} + \binom{N}{i+1}$.
\end{quote}
We let $\mathbf{Q}(i)$ be the following seemingly
weaker statement for a fixed $i\in\nats$.
\begin{quote}
$\mathbf{Q}(i)$ : For all nonnegative integers $m,m_1,m_2,N$
that satisfy $0\leq m\leq m_1\leq m_2 < \binom{N}{i}$ 
and $m_1+m_2 = m +\binom{N}{i}$ we have
$m_1^{(i)} + m_2^{(i)} \leq x^{(i)} + \binom{y_i+1}{i+1}$,
where 
\begin{equation}
\label{eqn:UBR}
m_1 = \binom{x_i}{i} + \cdots + \binom{x_p}{p}, \ \ 
m_2 = \binom{y_i}{i} + \cdots + \binom{y_q}{q}
\end{equation}
are their UBR, and $m_1+m_2 = x + \binom{y_i+1}{i}$ where $x\geq 0$.
\end{quote}
We now briefly argue the equivalence of $\mathbf{P}(i)$ and $\mathbf{Q}(i)$.

$\mathbf{P}(i)\Rightarrow \mathbf{Q}(i)$:
Let $i\in\nats$ be given. Suppose $0\leq m\leq m_1\leq m_2 < \binom{N}{i}$ 
and $m_1+m_2 = m +\binom{N}{i}$, and the UBR of $m_1$ and $m_2$ are given
as in (\ref{eqn:UBR}), then by assumption we have 
$m_2 < \binom{y_i+1}{i}\leq \binom{N}{i}$.
Hence $0\leq m_1,m_2\leq \binom{y_i+1}{i}$ and $x\geq 0$. By
$\mathbf{P}(i)$ we then obtain 
$m_1^{(i)} + m_2^{(i)} \leq x^{(i)} + \binom{y_i+1}{i+1}$, so we
have $\mathbf{Q}(i)$.

$\mathbf{Q}(i)\Rightarrow \mathbf{P}(i)$:
Let $i\in\nats$ be given. Suppose $0\leq m,m_1,m_2 \leq \binom{N}{i}$ 
and $m_1+m_2 = m +\binom{N}{i}$. If $m_2 = \binom{N}{i}$, then 
$\mathbf{P}(i)$ is trivially true. Also, by symmetry we may
assume that $m_1\leq m_2$, and so we may assume 
$0\leq m\leq m_1\leq m_2 < \binom{N}{i}$ and therefore we can apply
$\mathbf{Q}(i)$. Repeated use of $\mathbf{Q}(i)$, say $j\geq 1$ times,
will eventually yield
\[
m_1^{(i)} + m_2^{(i)} \leq x^{(i)} + \binom{y_i+j}{i+1}
\]
where $y_i+j = N$ and $x = m \geq 0$, which is $\mathbf{P}(i)$.

Therefore for each $i\in\nats$ the statements $\mathbf{P}(i)$
and $\mathbf{Q}(i)$ are equivalent.

\begin{proof} 
[Lemma~\ref{lmm:3}]
We will use induction and prove $\mathbf{P}(i)$ for all $i\in\nats$.
For $i=1$ we have $x^{(1)} = \binom{x}{2}$ for any integer $x$, 
so proving $\mathbf{P}(1)$ amounts to showing that 
$m_1^2 + m_2^2 \leq m^2 + N^2$ when $m,m_1,m_2\leq N$ and $m_1+m_2=m+N$ 
which is easily established.

We proceed by induction on $i$, assuming the equivalent
statements $\mathbf{P}(i-1)$ and $\mathbf{Q}(i-1)$, and 
prove $\mathbf{Q}(i)$.
Assume $0\leq m\leq m_1\leq m_2 < \binom{N}{i}$. Let
the UBR of $m_1$ and $m_2$ be as stated in (\ref{eqn:UBR}).
By assumption we have $x_i\leq y_i$. 

If $x' = m_1 - \binom{x_i}{i}$ and $y' = m_2 - \binom{y_i}{i}$,
assume for a moment that $x' \geq y'$. Then by the UBR of $m_1$
and $m_2$ we have $y_i\geq x_i > x_{i-1}\geq y_{i-1}$. If now 
$m_1' = \binom{x_i}{i} + y'$ and
$m_2' = \binom{y_i}{i} + x'$, then $m_1 + m_2 = m_1' + m_2'$
and $m_1'\leq m_1\leq m_2\leq m_2'$. Since 
$x_i > y_{i-1}$ and $y_i > x_{i-1}$ we have 
\begin{eqnarray*}
m_1^{(i)} + m_2^{(i)} 
& = & \binom{x_i}{i+1} + {x'}^{(i-1)} + \binom{y_i}{i+1} + {y'}^{(i-1)} \\
& = & \left( \binom{x_i}{i+1} + {y'}^{(i-1)}\right) 
  + \left( \binom{y_i}{i+1} + {x'}^{(i-1)}\right) \\
& = & {m_1'}^{(i)} + {m_2'}^{(i)}.
\end{eqnarray*}
Therefore we may further assume that $x'\leq y'$. We now consider
two cases.

{\sc First case} $x'+y'\geq \binom{y_i}{i-1}$:
In this case we have $x'+y' = x'' + \binom{y_i}{i-1}$ for some
$x''\geq 0$. Since $x'\leq y'$ we have $x_{i-1} < x_i\leq y_i$
and $y_{i-1}<y_i$ and hence from the UBR of $x'$ and $y'$ we
have $x'' < x'\leq y' < \binom{y_i}{i-1}$. By induction hypothesis
$\mathbf{Q}(i-1)$ we then have
${x'}^{(i-1)} + {y'}^{(i-1)} \leq {x''}^{(i-1)} + \binom{y_i}{i}$
and hence
\begin{eqnarray*}
m_1^{(i)} + m_2^{(i)} & = & 
  \left(\binom{x_i}{i} + x'\right)^{(i)} 
 +\left(\binom{y_i}{i} + y'\right)^{(i)} \\
                      & = & 
\binom{x_i}{i+1} + {x'}^{(i-1)} + \binom{y_i}{i+1} + {y'}^{(i-1)} \\
                      &\leq &
\binom{x_i}{i+1} + {x''}^{(i-1)} + \binom{y_i+1}{i+1}.
\end{eqnarray*}
If $x = \binom{x_i}{i} + x''$, then since $x'' < x'$, we
have $x^{(i)} =  \binom{x_i}{i+1} + {x''}^{(i-1)}$ and
$x + \binom{y_i+1}{i} = m_1+m_2$. From above we then have
$m_1^{(i)} + m_2^{(i)} \leq x^{(i)} + \binom{y_i+1}{i+1}$,
thereby obtaining $\mathbf{Q}(i)$ in this case.

{\sc Second case} $x'+y' < \binom{y_i}{i-1}$:
Note that for every $k\in\{1,\ldots,x_i\}$ we have 
\[
\binom{x_i}{i} = \binom{x_i-k}{i} + \sum_{\ell=1}^k\binom{x_i-\ell}{i-1}.
\]
By assumption of $\mathbf{Q}(i)$ we have $m_2 < \binom{N}{i}$
and $m_1+m_2 \geq \binom{N}{i}$ and hence by the UBR of $m_2$ we
have $m_1+m_2 \geq \binom{y_i+1}{i}$,
or $\binom{x_i}{i}+x'+y'\geq \binom{y_i}{i-1}$.
Therefore there is a unique $k\in\{1,\ldots,x_i\}$ such that
\[
\sum_{\ell=1}^k\binom{x_i-\ell}{i-1} + x'+y' \geq \binom{y_i}{i-1} >
\sum_{\ell=1}^{k-1}\binom{x_i-\ell}{i-1} + x'+y'.
\]
Hence 
$\sum_{\ell=1}^k\binom{x_i-\ell}{i-1} + x'+y' = \delta + \binom{y_i}{i-1}$
where $0\leq \delta < \binom{x_i-k}{i-1}$. Since $y_i \geq x_i > x_i-k$
we have further $\binom{y_i}{i-1} > \binom{x_i-k}{i-1}$ and hence
\[
0\leq \delta, \ \ \binom{x_i-k}{i-1}, \ \ 
\sum_{\ell=1}^{k-1}\binom{x_i-\ell}{i-1} + x'+y' < \binom{y_i}{i-1}.
\]
By Claim~\ref{clm:sum-less} and then by induction hypothesis $\mathbf{P}(i-1)$ 
we have
\begin{eqnarray*}
\sum_{\ell=1}^k\binom{x_i-\ell}{i} + {x'}^{(i-1)}+{y'}^{(i-1)}
 & = &  
\sum_{\ell=1}^k\binom{x_i-\ell}{i-1}^{(i-1)} + {x'}^{(i-1)}+{y'}^{(i-1)} \\
 & \leq &
\left(\sum_{\ell=1}^{k-1}\binom{x_i-\ell}{i-1} + x'+y'\right)^{(i-1)}
+ \binom{x_i-k}{i-1}^{(i-1)} \\
 & \leq & {\delta}^{(i-1)} + \binom{y_i}{i-1}^{(i-1)} \\
 & =    & {\delta}^{(i-1)} + \binom{y_i}{i}. \\
\end{eqnarray*}
Note that by definition of $\delta$ and its range, we have 
$m_1+m_2 = x + \binom{y_i+1}{i}$ where
$x = \binom{x_i-k}{i} + \delta$
and also
\[
x^{(i)} + \binom{y_i+1}{i}^{(i)}
= \binom{x_i-k}{i+1} + {\delta}^{(i-1)} + \binom{y_i+1}{i+1}.
\]
Since 
\[
\binom{x_i}{i+1} = \binom{x_i-k}{i+1} + \sum_{\ell=1}^k\binom{x_i-\ell}{i}
\]
we then finally get
\begin{eqnarray*}
m_1^{(i)} + m_2^{(i)} 
  & = & 
\binom{x_i}{i+1} + {x'}^{(i-1)} + \binom{y_i}{i+1} + {y'}^{(i-1)} \\
  & = & 
\binom{x_i-k}{i+1} + \sum_{\ell=1}^k\binom{x_i-\ell}{i} + 
{x'}^{(i-1)} + \binom{y_i}{i+1} + {y'}^{(i-1)} \\ 
  & \leq &
\binom{x_i-k}{i+1} + {\delta}^{(i-1)} + \binom{y_i}{i}
+ \binom{y_i}{i+1} \\
  & = & 
\binom{x_i-k}{i+1} + {\delta}^{(i-1)} + \binom{y_i+1}{i+1} \\
  & = & 
x^{(i)} + \binom{y_i+1}{i}^{(i)}
\end{eqnarray*}
which is $\mathbf{Q}(i)$. This completes the inductive proof that  
$\mathbf{P}(i-1)$ and $\mathbf{Q}(i-1)$ imply
$\mathbf{Q}(i)$, and so this completes the proof of of Lemma~\ref{lmm:3}.
\end{proof}

\section{The main theorem}
\label{sec:main}

In this section we use results from previous section
to prove our main result of this article Theorem~\ref{thm:main-full}
here below. 

Let $k\in\nats$ and $S\subseteq V(Q_k)$. 
Call a vertex/binary string of a subgraph $G = Q_k[S]$ of the 
$k$-dimensional hypercube $Q_k$ {\em full} if its degree is $k$ in $G$. 
For $n\in [2^k]$ Let $\phi_k(n)$ denote the maximum number of full 
vertices of an induced subgraph of $Q_k$ on $n$ vertices:
\[
\phi_k(n) = \max_{S\subseteq V(Q_k), |S|=n} 
|\{\tilde{x}\in S : d_{Q_k[S]}(\tilde{x}) = k\}.
\]
Clearly $\phi_k(2^k) = 2^k$ as every vertex of $Q_k$ is full.
If $n<2^k$ and $S\subseteq V(Q_k)$ contains $n$ vertices and
induces $\phi_k(n)$ full vertices in $Q_k$, then we can by symmetry of $Q_k$
(or relabeling of the vertices) assume that the vertex corresponding to the 
binary string consisting of $k$ 1's is not in $S$. In this case a 
vertex/string in $S$ with the maximum number of 1's is not full in $Q_k[S]$.
In particular we have $\phi_k(n) < n$ for each $n<2^k$.
\begin{observation}
\label{obs:2things}
For $k\in\nats$ we have:
\begin{enumerate}
  \item If $n < 2^k$ then $\phi_k(n) < n$.
  \item The function $\phi_k : [2^k] \rightarrow [2^k]$ is increasing.
\end{enumerate}
\end{observation}
{\sc Remark:} By Observation~\ref{obs:2things} we see that $\phi_k$ cannot be 
strictly increasing.

Note that every $n\in [2^k]$ has a unique {\em hypercube representation (HCR)}
as $n = \sum_{\ell = 0}^i\binom{k}{\ell} + m$ where $0\leq m < \binom{k}{i+1}$.
The main result of this section is the following.
\begin{theorem}
\label{thm:main-full}
For $k\in\nats$ and $n\in [2^k]$ with HCR 
$n =  \sum_{\ell = 0}^i\binom{k}{\ell} + m$, then 
\[
\phi_k(n) = \sum_{\ell = 0}^{i-1}\binom{k}{\ell} + m^{(k-i-1)}.
\]
\end{theorem}
We will prove Theorem~\ref{thm:main-full} by induction on $k$.
In order to do that, we will first derive a recursive upper bound
for $\phi_k(n)$.

Let $S \subseteq V(Q_k)$ be a set of $n$ vertices/binary strings,
and let $F_k(S)\subset S$ be the vertices of $S$ that are full
in $Q_k[S]$.
Looking at the decomposition $Q_k = Q_{k-1}^0 \boxplus Q_{k-1}^1$
let $S_b = S\cap V(Q_{k-1}^b)$ for $b=0,1$ and $n_b = |S_b|$.
Clearly $S = S_0 \cup S_1$ is a partition and we have $n_0 + n_1 = n$.
Note that for $b\in\{0,1\}$, a vertex in $S_b$ is full in $Q_k[S]$ iff 
(i) it is full in $Q_{k-1}^b[S_b]$, and 
(ii) its copy is contained in $S_{1-b}$.
By (i) and (ii) the number of vertices in $S_b$ 
that are full in $Q_k[S]$ is at most $\min(\phi_{k-1}(n_b),n_{1-b})$,
that is $|F_k(S)\cap S_b)|\leq \min(\phi_{k-1}(n_b),n_{1-b})$. Since
$F_k(S) = (F_k(S)\cap S_0)\cup (F_k(S)\cap S_1)$ 
is a partition we then have 
\[
|F_k(S)| = |F_k(S)\cap S_0| + |F_k(S)\cap S_1| \leq 
\min(\phi_{k-1}(n_0),n_1) + \min(\phi_{k-1}(n_1),n_0).
\]
By definition we then have the following recursive max-min 
upper bound 
\begin{equation}
\label{eqn:m+m}
\phi_k(n) \leq 
\max_{n_0+n_1=n}(\min(\phi_{k-1}(n_0),n_1) + \min(\phi_{k-1}(n_1),n_0)).
\end{equation}
Note that is impossible to have $n_b < \phi_{k-1}(n_{1-b})$ for both $b = 0,1$,
since then $n_0 < \phi_{k-1}(n_1) < n_1 < \phi_{k-1}(n_0) < n_0$,
a blatant contradiction. From this we see that (\ref{eqn:m+m})
can be written as 
\begin{equation}
\label{eqn:3m} 
\phi_k(n) \leq 
\max_{n_0+n_1=n}(\min(\phi_{k-1}(n_0) + \phi_{k-1}(n_1),
                n_0 + \phi_{k-1}(n_0),
                n_1 + \phi_{k-1}(n_1))).
\end{equation}
Further, by symmetry the maximum in (\ref{eqn:3m}) is attained
when $n_0\geq n_1$, in which case we have 
$n_1 + \phi_{k-1}(n_1) \leq n_0 + \phi_{k-1}(n_0)$. Hence we obtain
\begin{equation}
\label{eqn:phibound}
\phi_k(n) \leq 
\max_{n_0+n_1=n, n_0\geq n_1}(\min(\phi_{k-1}(n_0) + \phi_{k-1}(n_1),
               n_1 + \phi_{k-1}(n_1))).
\end{equation}
Let $f_k(n)$ be the function on the right in the displayed formula in 
the above Theorem~\ref{thm:main-full}
\[
f_k(n) := \sum_{\ell = 0}^{i-1}\binom{k}{\ell} + m^{(k-i-1)}
\]
where $n =  \sum_{\ell = 0}^i\binom{k}{\ell} + m$ is its HCR.
We first show that $\phi_k(n) \geq f_k(n)$ by 
explicitly show that an induced subgraph on 
$n$ vertices of $Q_k$ can have $f_k(n)$ full vertices.
Then we will show that $f_k(n)$ satisfies 
\begin{equation}
\label{eqn:feqn}
f_k(n) = \max_{n_0+n_1=n, n_0\geq n_1}(\min(f_{k-1}(n_0) + f_{k-1}(n_1),
n_1 + f_{k-1}(n_1))),
\end{equation}
which by (\ref{eqn:phibound}) shows that $f_k(n) \geq \phi_k(n)$.

For $k\in\nats$ let $n\in [2^k]$ with HCR
$n =  \sum_{\ell = 0}^i\binom{k}{\ell} + m$. To show that 
$\phi_k(n) \geq f_k(n)$ we construct an induced subgraph
of $Q_k$ on vertices with $f_k(n)$ full vertices as follows.
Let $S\subseteq V(Q_k)$ be the set of $n$ vertices containing
all $\sum_{\ell = 0}^i\binom{k}{\ell}$ binary strings having at most 
$i$ 1's in their representation, and the first $m$ binary strings 
with exactly $i+1$ 1's in their
representation in the lexicographical order. 
{\em Note!} Here a binary string represents the {\em opposite} 
subset of $[k]$; where the $j$-th bit
is 0 indicates that $j$ is included in the subset.
In this way the binary strings are ordered as their corresponding subsets 
of $[k]$ in the reverse lexicographical order. 
Clearly every vertex in 
the induced graph $Q_k[S]\subseteq Q_k$ with at most $i-1$ 1's
in their representation is full, these amount to 
$\sum_{\ell = 0}^{i-1}\binom{k}{\ell}$ full vertices. 
Also note that none of the $m$ vertices with exactly $i+1$ 1's in their
representation is full, as they are not connected to any vertex
with $i+2$ 1's in $Q_k[S]$. Among the $\binom{k}{i}$ binary strings
in $S$ containing exactly $i$ 1's, we briefly argue that $m^{(k-i-1)}$ 
of them are full in the following way. 

Consider the $(k-i-1)$-dimensional simplicial
complex $\Delta_{\tilde{f}}$ where 
\[
\tilde{f} = 
\left(\binom{k}{0},\ldots,\binom{k}{k-i-2},m,m^{(k-i-1)}\right)
\in {\ints}^{k-i+1}.
\]
Note that 
$\Delta_{\tilde{f}} \cap \left(\binom{k}{k-i}\cup \binom{k}{k-i-1}\right)$
is represented by the bipartite subgraph $G$ of $Q_k[S]$ induced by
the binary strings containing exactly $i$ or $i+1$ 1's, where
two stings are adjacent in $G$ iff for their opposite sets
the smaller one, with $k-i-1$ elements, is contained in the other
one with $k-i$ elements. 

Since each of the  $m^{(k-i-1)}$ subsets from 
$\binom{k}{k-i}\cap \Delta_{\tilde{f}}$
has all of its $k-i-1$ subsets among the $m$ subsets from 
$\Delta_{\tilde{f}}\cap \binom{k}{k-i-1}$, then the representing 
$m^{(k-i-1)}$ opposite binary strings in $G$, containing exactly $i$
1's, are each connected
to all the $k-i$ opposite binary strings among the $m$ ones in $G$,
that contain exactly $i+1$ 1's. 
Since each binary string in $G\subseteq Q_k[S]$ with $i$ 1's is
clearly connected to all $i$ binary strings with $i-1$ 1's in $Q_k[S]$,
we see that each of the mentioned $m^{(k-i-1)}$ opposite binary strings in 
$G\subseteq Q_k[S]$ are full. This shows that $Q_k[S]$ is an induced
subgraph of $Q_k$ with $n$ vertices and at least $f_k(n)$ full vertices.
Therefore we have $\phi_k(n) \geq f_k(n)$.

To complete the proof of Theorem~\ref{thm:main-full} we show
that $f_k(n)$ satisfies (\ref{eqn:feqn}), which by (\ref{eqn:phibound})
then implies that $\phi_k(n) \leq f_k(n)$, and hence $\phi_k(n) = f_k(n)$.
This will occupy the remainder of this section.
To show (\ref{eqn:feqn}), we will show
that 
$f_k(n) \geq \min(f_{k-1}(n_0) + f_{k-1}(n_1), n_1 + f_{k-1}(n_1))$,
whenever $n_0+n_1 = n$ and $n_0\geq n_1$. There are all together
six cases we will consider to verify this inequality; 
the first case (A) has two sub-cases (A1) and (A2), the second case (B)
has four sub-cases (B11), (B12), (B21) and (B22).

{\sc Case (A)} $f_{k-1}(n_0) \geq n_1$: Here we want to show
that $f_k(n)\geq n_1 + f_{k-1}(n_1)$. By definition of $f_k(n)$
we have here that $n_0 > f_{k-1}(n_0) \geq n_1$. Since $f_{k-1}$ is 
increasing there is a critical pair $(n_0^*,n_1^*)$ summing up
to $n$ such that (i) $f_{k-1}(n_0^*) \geq n_1^*$, and 
(ii) $f_{k-1}(n_0^*-1) < n_1^*+1$. Clearly we have 
$n_0\geq n_0^*$ and $n_1\leq n_1^*$, and so
$n_1 + f_{k-1}(n_1) \leq n_1^* + f_{k-1}(n_1^*)$. It therefore 
suffices to show that $f_k(n)\geq n_1^* + f_{k-1}(n_1^*)$.
Let $n =  \sum_{\ell = 0}^i\binom{k}{\ell} + m$ be its HCR.
Since $0\leq m < \binom{k}{i+1} = \binom{k-1}{i+1}+\binom{k-1}{i}$, 
we consider two sub-cases. 

{\sc Sub-case (A1)} $0\leq m < \binom{k-1}{i}$: Here in this
case we have a bipartition $n = n_0' + n_1'$ where 
\begin{equation}
\label{eqn:part-A1}
n_0' = \sum_{\ell=0}^i\binom{k-1}{\ell} , \ \ 
n_1' = \sum_{\ell=0}^{i-1}\binom{k-1}{\ell} + m 
\end{equation}
for which 
\begin{equation}
\label{eqn:cond-A1}
f_{k-1}(n_0') = \sum_{\ell=0}^{i-1}\binom{k-1}{\ell} \leq n_1'
\end{equation}
and hence, by definition of $n_0^*$ and $n_1^*$, we have $n_0^* \geq n_0'$,
$n_1^* \leq n_1'$ and so
\[
n_0^* = \sum_{\ell=0}^i\binom{k-1}{\ell} + m_0^*, \ \
n_1^* = \sum_{\ell=0}^{i-1}\binom{k-1}{\ell} + m_1^*,
\]
where $m_0^*, m_1^* \geq 0$ are integers, $m_0^* + m_1^* = m$,
and (i) ${m_0^*}^{(k-i-2)} \geq m_1^*$ and (ii) $(m_0^*-1)^{(k-i-2)} < m_1^*+1$.
Now note that $f_k(n)\geq n_1^* + f_{k-1}(n_1^*)$ is, by definition of $f_{k-1}$,
equivalent to $m^{(k-i-1)} \geq m_1^* + {m_1^*}^{(k-i-1)}$, which is implied
by $m^{(k-i-1)} \geq \min(m_1^*,{m_0^*}^{(k-i-2)}) + {m_1^*}^{(k-i-1)}$,
which holds by Lemma~\ref{lmm:1} since $m_0^* + m_1^* = m$.

{\sc Sub-case (A2)} $\binom{k-1}{i}\leq m < \binom{k}{i+1}$: Similarly
to Sub-case (A1) we have here in this case a bipartition $n = n_0' + n_1'$
where
\begin{equation}
\label{eqn:part-A2}
n_0' = \sum_{\ell=0}^i\binom{k-1}{\ell} + m', \ \ 
n_1' = \sum_{\ell=0}^{i}\binom{k-1}{\ell}
\end{equation}
where $m' = m - \binom{k-1}{i}$ for which 
\begin{equation}
\label{eqn:cond-A2}
f_{k-1}(n_0') = 
\sum_{\ell=0}^{i-1}\binom{k-1}{\ell} + {m'}^{(k-i-2)}
\leq n_1'
\end{equation}
and hence again, by definition of $n_0^*$ and $n_1^*$, we have 
$n_0^* \geq n_0'$, $n_1^* \leq n_1'$ and so
\[
n_0^* = \sum_{\ell=0}^i\binom{k-1}{\ell} + m_0^*, \ \
n_1^* = \sum_{\ell=0}^{i-1}\binom{k-1}{\ell} + m_1^*,
\]
where $m_0^* \geq m'$ and $m_1 < \binom{k-1}{i}$ are integers,
$m_0^* + m_1^* = m$, and (i) ${m_0^*}^{(k-i-2)} \geq m_1^*$ and 
(ii) $(m_0^*-1)^{(k-i-2)} < m_1^* + 1$.
Exactly as in the previous case (A1), we note that 
$f_k(n)\geq n_1^* + f_{k-1}(n_1^*)$ is by definition of $f_{k-1}$,
equivalent to $m^{(k-i-1)} \geq m_1^* + {m_1^*}^{(k-i-1)}$, which is implied
by $m^{(k-i-1)} \geq \min(m_1^*,{m_0^*}^{(k-i-2)}) + {m_1^*}^{(k-i-1)}$,
which again holds by Lemma~\ref{lmm:1} since $m_0^* + m_1^* = m$.

\comment{ 
m^{(k-i-1)} \geq \binom{k-1}{i} + \binom{k-1}{i-1} + m_1^* + {m_1^*}^{(k-i-2)}.
\end{equation}
Since $0\leq m' < \binom{k-1}{i+1} = \binom{k-1}{k-i-2}$ and 
$m = m' + \binom{k-1}{i} = m' + \binom{k-1}{k-i-1}$, we have
$m^{(k-i-1)} = \left(m' + \binom{k-1}{k-i-1}\right)^{(k-i-1)}
= \binom{k-1}{k-i} + {m'}^{(k-i-2)}$, and hence (\ref{eqn:A2}) is
equivalent to 
$(m_0^* + m_1^*)^{(k-i-2)} \geq \binom{k-1}{i} + m_1^* + {m_1^*}^{(k-i-2)}$,
which by assumption is equivalent to
\begin{equation}
\label{eqn:A2-2}
(m_0^* + m_1^*)^{(k-i-2)} \geq 
\min\left({m_0^*}^{(k-i-2)}, \binom{k-1}{i} + m_1^*\right) + {m_1^*}^{(k-i-2)}.
\end{equation}
The validity of (\ref{eqn:A2-2}) is now implied by Claim~\ref{clm:sum-less}.
}

{\sc Case (B)} $f_{k-1}(n_0) < n_1$: Here we want to show
that $f_k(n)\geq f_{k-1}(n_0) + f_{k-1}(n_1)$. By definition of $f_k(n)$
we have here that $n_0 \geq n_1 > f_{k-1}(n_0)$. 
Let $n =  \sum_{\ell = 0}^i\binom{k}{\ell} + m$ be its HCR.
Since $0\leq m < \binom{k}{i+1} = \binom{k-1}{i+1}+\binom{k-1}{i}$, 
we consider the two cases of whether $0\leq m< \binom{k-1}{i}$ or
$\binom{k-1}{i}\leq m < \binom{k}{i+1}$.

{\sc Sub-case (B1)} $0\leq m< \binom{k-1}{i}$: As in case (A1), we have a 
partition $n = n_0'+ n_1'$ given by (\ref{eqn:part-A1})
such that we have (\ref{eqn:cond-A1}). The two sub-cases here, (B11) and (B12),
depend on whether $n_0 \geq n_0'$ or $n_0 \leq n_0'$.

{\sc Sub-sub-case (B11)} $n_0\geq n_0'$ in (\ref{eqn:part-A1}):
Considering the critical pair $(n_0^*,n_1^*)$ from Case (A), 
we have here that $n_0' \leq n_0 < n_0^*$ and hence 
\[
n_0 = \sum_{\ell=0}^i\binom{k-1}{\ell} + m_0, \ \ 
n_1 = \sum_{\ell=0}^{i-1}\binom{k-1}{\ell} + m_1 
\]
where $0\leq m_0 < m_0*$, $m_1^* < m_1\leq m$, and $m_0+m_1=m$.
Now note that $f_k(n)\geq f_{k-1}(n_0) + f_{k-1}(n_1)$ is 
by definition of $f_{k-1}$, equivalent to 
\begin{equation}
\label{eqn:B11}
m^{(k-i-1)} \geq m_0^{(k-i-2)} + m_1^{(k-i-1)}.
\end{equation} 
By definition of $m_0^*$ we have 
$m_0^{(k-i-2)}\leq {m_0^*}^{(k-i-2)} \leq m_1^* < m_1$ and hence
(\ref{eqn:B11}) is equivalent to 
$(m_1+m_0)^{(k-i-1)} \geq m_1^{(k-i-1)} + \min(m_0^{(k-i-2)},m_1)$,
which is implied by Lemma~\ref{lmm:1}.

{\sc Sub-sub-case (B12)} $n_0 \leq n_0'$ in (\ref{eqn:part-A1}):
Here we then have $n/2\leq n_0 \leq n_0'$ and $n_1'\leq n_1\leq n/2$, and 
hence 
\[
n_0 = \sum_{\ell=0}^{i-1}\binom{k-1}{\ell} + m_0, \ \ 
n_1 = \sum_{\ell=0}^{i-1}\binom{k-1}{\ell} + m_1 
\]
where $0\leq m_1 \leq m_0$ and $m_0+m_1 = m +\binom{k-1}{i}$,
and hence 
$\left(m+\binom{k-1}{i}\right){\big /}2\leq m_1 \leq m_0 < \binom{k-1}{i}$.
Here $f_k(n)\geq f_{k-1}(n_0) + f_{k-1}(n_1)$ is 
by definition of $f_{k-1}$, equivalent to 
$\binom{k-1}{i-1} + m^{(k-i-1)} \leq m_0^{(k-i-1)} + m_0^{(k-i-1)}$,
which holds by Lemma~\ref{lmm:3}. 

{\sc Sub-case (B2)} $\binom{k-1}{i}\leq m < \binom{k}{i+1}$: 
As in case (A2), we have a 
partition $n = n_0'+ n_1'$ given by (\ref{eqn:part-A2})
such that we have (\ref{eqn:cond-A2}). As in the case (B1),
the two sub-cases here, (B21) and (B22),
depend on whether $n_0 \geq n_0'$ or $n_0 \leq n_0'$.

{\sc Sub-sub-case (B21)} $n_0\geq n_0'$ in (\ref{eqn:part-A2}):
Considering the critical pair $(n_0^*,n_1^*)$ from Case (A), 
we have here that $n_0' \leq n_0 < n_0^*$ and hence 
\[
n_0 = \sum_{\ell=0}^i\binom{k-1}{\ell} + m_0, \ \ 
n_1 = \sum_{\ell=0}^{i-1}\binom{k-1}{\ell} + m_1 
\]
where $m' := m-\binom{k-1}{i}\leq m_0 < m_0*$, 
$m_1^* < m_1\leq \binom{k-1}{i}$, and $m_0+m_1=m$.
Now note that $f_k(n)\geq f_{k-1}(n_0) + f_{k-1}(n_1)$ is 
by definition of $f_{k-1}$, equivalent to 
\begin{equation}
\label{eqn:B21}
m^{(k-i-1)} \geq m_0^{(k-i-2)} + m_1^{(k-i-1)}.
\end{equation}
Since $m_0 \leq m_0^*-1$ we have by definition of $m_0^*$
that 
$m_0^{(k-i-2)}\leq (m_0^*-1)^{(k-i-2)} < m_1^*+1 \leq m_1$
and hence (\ref{eqn:B21}) is equivalent to 
$(m_1+m_0)^{(k-i-1)} \geq m_1^{(k-i-1)} + \min(m_0^{(k-i-2)},m_1)$,
which is implied by Lemma~\ref{lmm:1}.

{\sc Sub-sub-case (B22)} $n_0\leq n_0'$ in (\ref{eqn:part-A2}):
Here we then have $n/2\leq n_0 \leq n_0'$ and $n_1'\leq n_1\leq n/2$, and 
hence 
\[
n_0 = \sum_{\ell=0}^{i}\binom{k-1}{\ell} + m_0, \ \ 
n_1 = \sum_{\ell=0}^{i}\binom{k-1}{\ell} + m_1 
\]
where $m'/2 \leq m_0 \leq m'$ and 
$0\leq m_1\leq m'/2$, and $m_0+m_1 = m' = m - \binom{k-1}{i}$.
Here $f_k(n)\geq f_{k-1}(n_0) + f_{k-1}(n_1)$ is 
by definition of $f_{k-1}$, equivalent to 
$m^{(k-i-1)} \geq m_0^{(k-i-2)} + m_0^{(k-i-2)} + \binom{k-1}{i-1}$.
Since $m = m_0 + m_1 + \binom{k-1}{k-i-1}$ this follows from
Lemma~\ref{lmm:2}.

In all the above six cases (A1), (A2), and (B11), (B12), (B21) and (B22),
we have that 
$f_k(n) \geq \min(f_{k-1}(n_0) + f_{k-1}(n_1), n_1 + f_{k-1}(n_1))$
whenever $n_0+n_1 = n$ and $n_0\geq n_1$. This shows that $f_k(n)$ 
satisfies (\ref{eqn:feqn}) and therefore that $\phi_k(n) \leq f_k(n)$,
which completes the proof of Theorem~\ref{thm:main-full}.

\section{An application}
\label{sec:appl}

In this section we apply the main result of the previous
section, Theorem~\ref{thm:main-full}, to determine the 
value $\min(\max( |V(H_1)|, |V(H_2)|))$  
where (i) $H_1$ and $H_2$ are induced subgraphs of $Q_k$, and
(ii) together $H_1$ and $H_2$ cover all the edges of $Q_k$. The main
(and the only) theorem in this section is the following.
\begin{theorem}
\label{thm:min-max-cover}
For $k\in\nats$ we have
\[
\min_{E(H_1)\cup E(H_2) = E(Q_k)}(\max( |V(H_1)|, |V(H_2)|)) = 
\sum_{\ell = 0}^{\lfloor k/2\rfloor }\binom{k}{\ell} 
+ ( k\bmod 2 )\binom{k-1}{\lfloor k/2\rfloor}.
\]
\comment{ 
\[
\min_{E(H_1)\cup E(H_2) = E(Q_k)}(\max( |V(H_1)|, |V(H_2)|)) = 
\left\{ 
\begin{array}{ll}
  \sum_{\ell = 0}^{k/2}\binom{k}{\ell} & \mbox{ if $k$ is even, } \\
  2\left(
\sum_{\ell = 0}^{(k-1)/2}\binom{k-1}{\ell}\right)  & \mbox{ if $k$ is odd. }
\end{array}
\right.
\]
}
\end{theorem}
The rest of this final section will be devoted
to prove Theorem~\ref{thm:min-max-cover}.

Assume $k$ is even and that $|V(H_1)| < \sum_{\ell = 0}^{k/2}\binom{k}{\ell}$.
In this case we have
\[
|V(H_1)| \leq \sum_{\ell = 0}^{k/2}\binom{k}{\ell} - 1 
=             \sum_{\ell = 0}^{k/2 - 1}\binom{k}{\ell} 
+ \left( \binom{k}{k/2} - 1\right).
\]
By Observation~\ref{obs:2things}, Theorem~\ref{thm:min-max-cover} and
Observation~\ref{obs:strict} we obtain the following.
\begin{eqnarray*}
\phi_k(|V(H_1)|) &\leq & 
  \phi_k\left(  \sum_{\ell = 0}^{k/2 - 1}\binom{k}{\ell} 
+ \left( \binom{k}{k/2} - 1\right)\right) \\
  & = & 
\sum_{\ell = 0}^{k/2 - 2}\binom{k}{\ell} 
+ \left( \binom{k}{k/2} - 1\right)^{(k/2)} \\
  & = & 
\sum_{\ell = 0}^{k/2 - 2}\binom{k}{\ell} 
+ \binom{k}{k/2 + 1} - k/2 \\ 
  & = & 
\sum_{\ell = 0}^{k/2 - 1}\binom{k}{\ell} - k/2 \\ 
  & < & 
\sum_{\ell = 0}^{k/2 - 1}\binom{k}{\ell}. 
\end{eqnarray*}
Since every vertex in $Q_k$ that is not full in $H_1$ 
is incident to an edge in $H_2$ and is therefore a vertex
in $H_2$ we have that 
\[
|V(H_2)| 
\geq |Q_k| - \phi_k(|V(H_1)|) 
> 2^k - \sum_{\ell = 0}^{k/2-1}\binom{k}{\ell}
=       \sum_{\ell = 0}^{k/2}\binom{k}{\ell}
\]
and hence $\max( |V(H_1)|, |V(H_2)|) > \sum_{\ell = 0}^{k/2}\binom{k}{\ell}$.
On the other hand, if $H_1$ and $H_2$ are the subgraph of $Q_k$ 
induced by binary strings of length $k$ with at most $k/2$ 0's
and with at most $k/2$ 1s respectively, then 
$|V(H_1)| = |V(H_2)| = \sum_{\ell = 0}^{k/2}\binom{k}{\ell}$ and hence
$\max( |V(H_1)|, |V(H_2)|) = \sum_{\ell = 0}^{k/2}\binom{k}{\ell}$.
Hence, as $H_1$ and $H_2$ cover all the edges of $Q_k$, then 
Theorem~\ref{thm:min-max-cover} is valid for even $k$.

Assume $k$ is odd and that 
\[
|V(H_1)| < \sum_{\ell = 0}^{\lfloor k/2\rfloor}\binom{k}{\ell} 
+ \binom{k-1}{\lfloor k/2\rfloor},
\]
and hence 
\[
|V(H_1)| \leq 
\sum_{\ell = 0}^{(k-1)/2}\binom{k}{\ell} + 
\left(\binom{k-1}{\frac{k-1}{2}} - 1\right).
\]
As in the even case, we obtain here by 
Observation~\ref{obs:2things}, Theorem~\ref{thm:min-max-cover} and
Observation~\ref{obs:strict} that 
\begin{eqnarray*}
\phi_k(|V(H_1)|) &\leq & 
\phi_k\left(\sum_{\ell = 0}^{(k-1)/2}\binom{k}{\ell} + 
\left(\binom{k-1}{\frac{k-1}{2}} - 1\right)\right) \\
  & = & \sum_{\ell = 0}^{(k-3)/2}\binom{k}{\ell} + 
\left(\binom{k-1}{\frac{k-1}{2}} - 1\right)^{\left(\frac{k-1}{2}\right)} \\
  & = & \sum_{\ell = 0}^{(k-3)/2}\binom{k}{\ell} + 
\binom{k-1}{\frac{k+1}{2}} - \frac{k-1}{2}\\
  & < & \sum_{\ell = 0}^{(k-3)/2}\binom{k}{\ell} + 
\binom{k-1}{\frac{k+1}{2}}. 
\end{eqnarray*}
Again, since every vertex in $Q_k$ that is not full in $H_1$ 
is incident to an edge in $H_2$ and is therefore a vertex
in $H_2$ we have that 
\[
|V(H_2)| 
\geq |Q_k| - \phi_k(|V(H_1)|) 
> 2^k - \sum_{\ell = 0}^{(k-3)/2}\binom{k}{\ell}
- \binom{k-1}{\frac{k+1}{2}}
= \sum_{\ell = 0}^{(k-1)/2}\binom{k}{\ell} 
+ \binom{k-1}{\frac{k-1}{2}}
\]
and hence 
\[
\max( |V(H_1)|, |V(H_2)|) > \sum_{\ell = 0}^{\lfloor k/2\rfloor}\binom{k}{\ell} 
+ \binom{k-1}{\lfloor k/2\rfloor}.
\]
On the other hand, considering the subgraphs $H_1$ and $H_2$ of $Q_k$ 
induced by binary strings of length $k$, where
$H_1$ is induced by the strings with at most $(k-1)/2$ 1's 
among the first $k-1$ bits, and 
$H_2$ is induced by the strings with at most $(k-1)/2$ 0's
among the first $k-1$ bits, we have that  
$|V(H_1)| = |V(H_2)| = 2\left( \sum_{\ell = 0}^{(k-1)/2}\binom{k-1}{\ell}\right)$
and hence 
\[
\max( |V(H_1)|, |V(H_2)|) = 
2\left( \sum_{\ell = 0}^{(k-1)/2}\binom{k-1}{\ell}\right)
=  \sum_{\ell = 0}^{\lfloor k/2\rfloor}\binom{k}{\ell} 
+ \binom{k-1}{\lfloor k/2\rfloor}.
\]
Hence, as $H_1$ and $H_2$ cover all the edges of $Q_k$, then
Theorem~\ref{thm:min-max-cover} is valid for odd $k$.
This completes the proof of Theorem~\ref{thm:min-max-cover}.

\subsection*{Acknowledgments}  

Sincere thanks to the anonymous referees for ...

\flushright{\today}

\end{document}